\documentclass[secnum,leqno]{rims-bessatsu_web}

\usepackage{amssymb,amsmath}
\usepackage{graphics}
\usepackage[dvips]{graphicx}
\usepackage{eepic}
\usepackage{epic}

\newtheorem{thm}{Theorem}[section]
\newtheorem{crl}[thm]{Corollary}

\newtheorem{lmm}[thm]{Lemma}

\newtheorem{prp}[thm]{Proposition}

\theoremstyle{definition}

\newtheorem{exa}[thm]{Example}
\theoremstyle{remark}

\title{Linear stability of the incoherent solution and the transition formula for the Kuramoto-Daido model}
\TitleHead{Linear stability of the incoherent solution for the Kuramoto-Daido model}          
\author{Hayato \textsc{Chiba}\footnote{ Faculty of Mathematics, Kyushu University, Fukuoka,
819-0395, Japan.\newline e-mail: \texttt{chiba@math.kyushu-u.ac.jp}}
          }
\AuthorHead{Hayato Chiba}         

\begin{document}

\maketitle
\vspace*{-1.0cm}
\begin{center}
October, 19, 2009
\end{center}

\begin{abstract}      
The Kuramoto-Daido model, which describes synchronization phenomena,
is a system of ordinary differential equations on $N$-torus defined as coupled harmonic oscillators,
whose natural frequencies are drawn from some distribution function.
In this paper, the continuous model for the Kuramoto-Daido model is introduced and 
the linear stability of its trivial solution (incoherent solution) is investigated.
Kuramoto's transition point $K_c$, at which the incoherent solution changes the stability, 
is derived for an arbitrary distribution function for natural frequencies.
It is proved that if the coupling strength $K$ is smaller than $K_c$,
the incoherent solution is asymptotically stable, while if $K$ is larger than $K_c$,
it is unstable.
\end{abstract}


\section{Introduction}

Collective synchronization phenomena are observed in a variety of areas such as chemical reactions,
engineering circuits and biological populations~\cite{Pik}.
In order to investigate such a phenomenon, Kuramoto~\cite{Kura1} proposed a system of ordinary differential equations
\begin{equation}
\frac{d\theta _i}{dt} 
= \omega _i + \frac{K}{N} \sum^N_{j=1} \sin (\theta _j - \theta _i),\,\, i= 1, \cdots  ,N,
\label{KMN}
\end{equation}
where $\theta _i \in [ 0, 2\pi )$ denotes the phase of an $i$-th oscillator on a circle,
$\omega _i\in \mathbf{R}$ denotes its natural frequency, $K>0$ is the coupling strength,
and where $N$ is the number of oscillators.
Eq.(\ref{KMN}) is derived by means of the averaging method from coupled dynamical systems having 
limit cycles, and now it is called the \textit{Kuramoto model}.

It is obvious that when $K=0$, $\theta _i(t)$ and $\theta _j(t)$ rotate on a circle at 
different velocities unless $\omega _i$ is equal to $\omega _j$, 
and it is true for sufficiently small $K>0$.
On the other hand, if $K$ is sufficiently large, it is numerically observed that
some of oscillators or all of them tend to rotate at the same velocity on average, which is called the 
\textit{synchronization}~\cite{Pik,Str1,Mir1}.
If $N$ is small, such a transition from de-synchronization to synchronization may be well revealed
by means of the bifurcation theory~\cite{ChiPa,Mai1,Mai2}.
However, if $N$ is large, it is difficult to investigate the transition from the view point of
the bifurcation theory and it is still far from understood.

In order to evaluate whether synchronization occurs or not, Kuramoto introduced
the \textit{order parameter} $r(t)e^{\sqrt{-1}\psi (t)}$ by
\begin{equation}
r(t)e^{\sqrt{-1}\psi (t)} := \frac{1}{N}\sum^N_{j=1} e^{\sqrt{-1} \theta _j(t)},
\label{order2}
\end{equation}
which gives the centroid of oscillators, where $r, \psi \in \mathbf{R}$.
It seems that if synchronous state is formed, $r(t)$ takes a positive number, while
if de-synchronization is stable, $r(t)$ is zero on time average (see Fig.\ref{fig1}).
Based on this observation and some formal calculation, Kuramoto conjectured a bifurcation diagram
of $r(t)$ as follows:
\\[0.2cm]
\textbf{Kuramoto's conjecture}

Suppose that $N\to \infty$ and natural frequencies $\omega _i$'s are distributed according to a probability
density function $g(\omega )$.
If $g(\omega )$ is an even and unimodal function, then the bifurcation
diagram of $r(t)$ is given as Fig.\ref{fig2} (a); that is, if the coupling strength $K$ is 
smaller than $K_c := 2/(\pi g(0))$, then $r(t) \equiv 0$ is asymptotically stable.
On the other hand, if $K$ is larger than $K_c$, there exists a positive constant $r_c$ such that 
$r(t) = r_c$ is asymptotically stable.
Near the transition point $K_c$, the scaling law of $r_c$ is of $O((K- K_c)^{1/2})$.

\begin{figure}
\begin{center}
\includegraphics[]{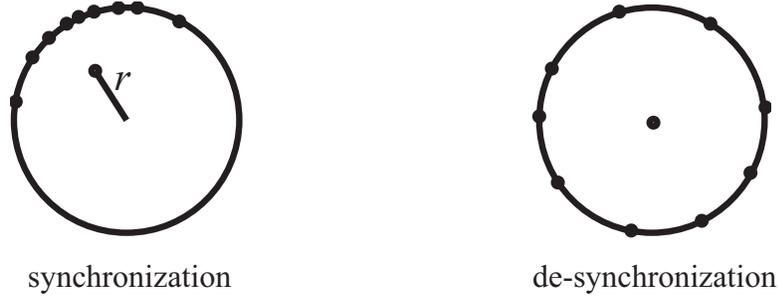}
\caption[]{The order parameter for the Kuramoto model.}
\label{fig1}
\end{center}
\end{figure}

\begin{figure}
\begin{center}
\includegraphics[]{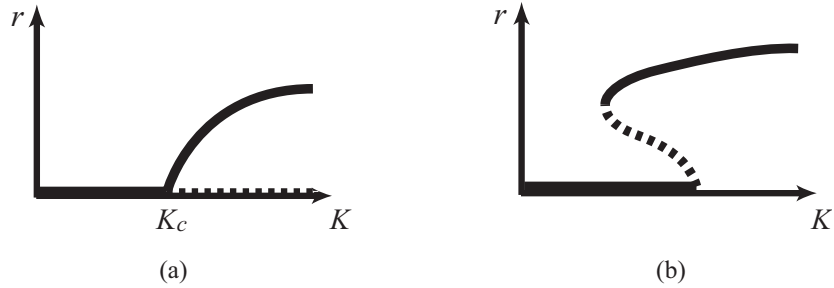}
\caption[]{Typical bifurcation diagrams of the order parameter for the cases that (a) $g(\omega )$
is even and unimodal (b) $g(\omega )$ is even and bimodal.
Solid lines denote stable solutions and dotted lines denote unstable solutions.}
\label{fig2}
\end{center}
\end{figure}

Now the value $K_c = 2/(\pi g(0))$ is called the \textit{Kuramoto's transition point}.
See \cite{Kura2} and \cite{Str1} for the Kuramoto's discussion.

Significant papers of Strogatz et al.~\cite{Str2, Str3, Mir2} partially confirmed the Kuramoto's conjecture.
Though their arguments are not rigorous from a mathematical view point, almost all of them are
justified as will be done in this paper.
In \cite{Str3}, they introduced the continuous model for the Kuramoto model and investigated the linear stability
of a trivial solution called the \textit{incoherent solution}, which corresponds to the de-synchronous state $r\equiv 0$.
They derived the Kuramoto's transition point $K_c = 2/(\pi g(0))$ and showed that if $K > K_c$,
the incoherent solution is unstable in the linear level (i.e. nonlinear terms are neglected).
When $K < K_c$, the linear operator $T$, which defines the linearized equation
of the continuous model around the incoherent solution, has no eigenvalues.
However, in \cite{Str2}, they found that
an analytic continuation of the resolvent $(\lambda -T)^{-1}$ may have poles (\textit{resonance poles})
on the left half plane, and they remarked a possibility that resonance poles induce exponential decay
of the order parameter.
In \cite{Mir2}, the stability of the partially locked state, which corresponds to a solution with positive constant $r=r_c$,
is investigated in the linear level.

Despite the active interest in the case that the distribution function $g(\omega )$ is even and unimodal,
bifurcation diagrams of $r$ for $g(\omega )$ other than the even and unimodal cases are not revealed well.
Martens et al.~\cite{Mar} investigated the bifurcation diagram for a bimodal $g(\omega )$
which consists of two Lorentzian distributions.
In particular, they found that stable partially locked states can coexist with stable 
incoherent solutions if $K$ is slightly smaller than $K_c$ (see Fig.\ref{fig2} (b)).
Such a diagram seems to be common for any bimodal distributions.

A simple extension of the Kuramoto model defined to be
\begin{equation}
\frac{d\theta _i}{dt} 
= \omega _i + \frac{K}{N} \sum^N_{j=1} f(\theta _j - \theta _i),\,\, i= 1, \cdots  ,N,
\label{KDMN}
\end{equation}
is called the \textit{Kuramoto-Daido model} \cite{Daido1,Daido2,Daido3,Daido4}, where the $2\pi$-periodic function 
$f : \mathbf{R} \to \mathbf{R}$ is called the \textit{coupling function}.
Daido~\cite{Daido4} investigated bifurcation diagrams of the order parameter for the Kuramoto-Daido model
with even and unimodal $g(\omega )$ by a similar argument to Kuramoto's one.
He found that if $f(\theta ) \neq \sin \theta $, partially locked states may coexist with stable 
incoherent solutions even if $g(\omega )$ is even and unimodal.

All such studies by physicist are based on formal calculations and numerical simulations.
The purpose of this paper is to justify and extend their results as mathematics for
the Kuramoto-Daido model with any distribution function $g(\omega )$. 
The continuous model for the Kuramoto-Daido model is introduced and 
the linear stability of the incoherent solution is studied.
In particular, the spectrum and the semigroup of a linear operator $T$, which is 
obtained by linearizing the continuous model around the incoherent solution, will be investigated in detail.
At first, a formula for obtaining the transition point $K_c$ for an arbitrary distribution $g(\omega )$ is derived.
As a corollary, the Kuramoto's transition point $K_c = 2/(\pi g(0))$ is obtained if $g(\omega )$ 
is an even and unimodal function.
If $K > K_c$, it is proved that the incoherent solution is unstable
because the operator $T$ has eigenvalues on the right half plane.
It means that if the coupling strength $K$ is large, the de-synchronous state is unstable 
and thus synchronization may occur.
On the other hand, if $0< K < K_c$, 
it will be shown that the spectrum of the operator $T$ consists of the continuous spectrum
and it lies on the imaginary axis. Thus the stability of the incoherent solution is nontrivial.
Despite this fact, under appropriate assumptions for $g(\omega )$,
the order parameter proves to decay exponentially because of existence of resonance poles on the left half plane
as was expected by Strogatz et al.~\cite{Str2}.
It suggests that in general, linear stability of a trivial solution of a linear equation on an
infinite dimensional space is determined by not only the spectrum of the linear operator but 
also its resonance poles.


\section{Continuous model}

In this section, we introduce a continuous model of the Kuramoto-Daido model
and show a few properties of it.

Let us consider the Kuramoto-Daido model (\ref{KDMN}). 
We suppose that the coupling function $f$ is a $C^1$ periodic function with the period $2\pi$.
It is expanded in a Fourier series as
\begin{equation}
f(\theta ) = \sum^{\infty}_{l=-\infty} f_l\, e^{\sqrt{-1}l \theta },\,\, f_l \in \mathbf{C}.
\label{2-1}
\end{equation} 
We can suppose that $f_0 = 0$ without loss of generality because $f_0$ is renormalized into the constants $\omega _i$.
For the Kuramoto model ($f(\theta ) = \sin \theta $), $f_{\pm 1} = \pm 1/(2\sqrt{-1})$ and $f_l = 0 \,\, (l \neq \pm 1)$.
Following Daido~\cite{Daido4}, we introduce the \textit{generalized order parameters} $\hat{Z}^{0}_k$ by
\begin{equation}
\hat{Z}^{0}_k (t) = \frac{1}{N} \sum^N_{j=1} e^{\sqrt{-1} k\theta _j(t)}, \,\, k=0, \pm 1, \pm 2, \cdots .
\label{order}
\end{equation}
In particular, $\hat{Z}^0_1$ is the order parameter defined in Section 1.
By using them, Eq.(\ref{KDMN}) is rewritten as
\begin{equation}
\frac{d \theta _i}{dt} = \omega _i + K \sum^\infty_{l=-\infty} f_l \hat{Z}^{0}_l (t) e^{-\sqrt{-1} l \theta _i}.
\end{equation}
Motivated by these equations, we introduce a continuous model of the Kuramoto-Daido model,
which is an evolution equation of a probability density function $\rho_t = \rho_t (\theta , \omega )$ 
on $S^1 \times \mathbf{R}$ parameterized by $t \in \mathbf{R}$, as
\begin{eqnarray}
\left\{ \begin{array}{ll}
\displaystyle \frac{\partial \rho_t}{\partial t} + 
\frac{\partial }{\partial \theta }
\left( \Bigl(\omega  + K \sum^\infty_{l=-\infty} f_l Z^0_l (t) e^{-\sqrt{-1} l \theta }\Bigr) \rho_t \right) = 0,  \\[0.3cm]
\displaystyle Z^0_k(t) = \int_{\mathbf{R}} \! \int^{2\pi}_{0} \!
 e^{\sqrt{-1} k \theta } \rho_t (\theta , \omega ) d\theta d\omega ,  \\[0.3cm]
\rho_0 (\theta , \omega ) = h(\theta , \omega ),
\end{array} \right.
\label{conti}
\end{eqnarray}
where $h(\theta , \omega )$ is an initial density function.
The $Z^0_k(t)$ is a continuous version of $\hat{Z}^{0}_k (t)$, and we also call it the \textit{generalized
order parameter}.
We can prove that Eq.(\ref{conti}) is proper in the sense that $\hat{Z}^{0}_k (t) \to Z^0_k(t)$ as $N \to \infty$
under some assumptions, although the proof is not given in this paper.
If we regard 
\begin{eqnarray*}
v := \omega  + K \sum^\infty_{l=-\infty} f_l Z^0_l (t) e^{-\sqrt{-1}l \theta }
\end{eqnarray*}
as a velocity field, Eq.(\ref{conti}) provides an equation of continuity 
$\partial \rho_t / \partial t + \partial (\rho_t v)/ \partial \theta  = 0$ known in fluid dynamics.
It is easy to prove the low of conservation of mass:
\begin{equation}
\int_{\mathbf{R}} \! \int^{2\pi}_{0} \! \rho_t (\theta , \omega ) d\theta
 =  \int_{\mathbf{R}} \! \int^{2\pi}_{0} \! h(\theta , \omega ) d\theta  =: g(\omega ).
\label{g}
\end{equation}
A function $g$ defined as above gives a probability density function for natural frequencies
$\omega \in \mathbf{R}$ such that $\int_{\mathbf{R}} \! g(\omega )d\omega  = 1$.

By using the characteristic curve method, Eq.(\ref{conti}) is formally integrated as follows:
Consider the equation
\begin{eqnarray}
\frac{dx}{dt} &=& \omega + K \sum^\infty_{l=-\infty} f_l Z^0_l (t) e^{-\sqrt{-1} l x },\,\, x\in S^1,
\label{cha0} 
\end{eqnarray}
which defines a characteristic curve.
Let $x = x(t, s; \theta,\omega  )$ be a solution of Eq.(\ref{cha0}) satisfying $x(s,s; \theta,\omega  ) = \theta $.
Then, $\rho_t$ is given as
\begin{equation}
\rho_t (\theta , \omega ) = h(x(0,t; \theta,\omega ), \omega ) \exp \Bigl[ 
K \int^t_{0} \! \sum^\infty_{l=-\infty}\sqrt{-1}\,l\, f_l\, Z^0_l (s) e^{-\sqrt{-1} l x(s,t; \theta,\omega ) }ds \Bigr]. 
\label{cha2}
\end{equation}
By using Eq.(\ref{cha2}), it is easy to show the equality
\begin{equation}
\int_{\mathbf{R}} \! \int^{2\pi}_{0} \! a(\theta , \omega ) \rho_t (\theta , \omega ) d\theta d\omega
 =  \int_{\mathbf{R}} \! \int^{2\pi}_{0} \! a(x(t,0; \theta,\omega ) , \omega ) h(\theta , \omega ) d\theta d\omega,
\label{cha3}
\end{equation}
for any continuous function $a(\theta , \omega )$.
In particular, the generalized order parameters $Z^0_k (t)$ are rewritten as
\begin{equation}
Z^0_k (t)
 =  \int_{\mathbf{R}} \! \int^{2\pi}_{0} \! e^{\sqrt{-1} k x(t,0; \theta, \omega)} h (\theta , \omega ) d\theta d\omega.
\label{cha4}
\end{equation}
Substituting it into Eqs.(\ref{cha0}), (\ref{cha2}), we obtain
\begin{equation}
\frac{d}{dt}x(t,s; \theta , \omega )
 = \omega  + K \int_{\mathbf{R}} \! \int^{2\pi}_{0} \! 
       f(x(t,0;\theta ', \omega ') - x(t,s; \theta , \omega )) h (\theta' , \omega' ) d\theta' d\omega',
\label{sol1}
\end{equation}
and
\begin{eqnarray}
\rho_t (\theta , \omega ) &=&  h(x(0,t; \theta,\omega ), \omega ) \times \nonumber \\ 
& & \exp \Bigl[ 
K \int^t_{0} \! ds \cdot \int_{\mathbf{R}} \! \int^{2\pi}_{0} \!\frac{\partial f}{\partial \theta }
\Bigl( x(s,0;\theta ', \omega ') - x(s,t; \theta , \omega ) \Bigr) h(\theta' , \omega' ) d\theta' d\omega' \Bigr],
\label{sol2}
\end{eqnarray}
respectively.
Even if $h(\theta , \omega )$ is not differentiable, we consider Eq.(\ref{sol2})
to be a weak solution of Eq.(\ref{conti}).
It is easy in usual way to prove that the integro-ODE (\ref{sol1}) has a unique solution 
for any $t>0$, and this proves that
the continuous model Eq.(\ref{conti}) has a unique weak solution (\ref{sol2}) for 
an arbitrary initial data $h(\theta , \omega )$.

Throughout this paper, we suppose that the initial date $h(\theta , \omega )$ is of the form
$h(\theta , \omega ) = \hat{h}(\theta ) g(\omega )$.
This assumption corresponds to the assumption for the Kuramoto-Daido model (\ref{KDMN}) that
initial values $\{\theta _j(0)\}_{j=1}^{N}$ and natural frequencies $\{\omega _j\}_{j=1}^{N}$
are independently distributed.
This is a physically natural assumption used in many literatures.
In this case, $\rho_t(\theta , \omega )$ is written as $\rho_t(\theta , \omega )
 = \hat{\rho}_t(\theta , \omega ) g(\omega )$, where
\begin{eqnarray}
\hat{\rho}_t(\theta , \omega ) &=& \hat{h}(x(0,t; \theta,\omega )) \times \nonumber \\
& &  \exp \Bigl[ 
K \int^t_{0} \! ds \cdot \int_{\mathbf{R}} \! \int^{2\pi}_{0} \!\frac{\partial f}{\partial \theta }
\Bigl( x(s,0;\theta ', \omega ') - x(s,t; \theta , \omega ) \Bigr) \hat{h}(\theta')g(\omega' ) d\theta' d\omega' \Bigr],
\end{eqnarray}
and $\hat{\rho}_t(\theta , \omega )$ satisfies the same equation as Eq.(\ref{conti}).


\section{Linear stability of the incoherent solution}

A trivial solution of the continuous model (\ref{conti}),
which is independent of $\theta $ and $t$, is given by $\rho_t (\theta , \omega ) = g(\omega )/(2\pi)$,
or equivalently $\hat{\rho}_t (\theta , \omega ) = 1/(2\pi)$.
It is called the \textit{incoherent solution}, which corresponds to the de-synchronized state.
Note that in this case $r=0$.
In this section, we investigate the stability of the incoherent solution and the order parameter.

Let
\begin{equation}
Z_j(t, \omega ) := \int^{2\pi}_{0} \! e^{\sqrt{-1}j\theta }  \hat{\rho}_t (\theta , \omega ) d\theta 
\end{equation}
be the Fourier coefficients of $\hat{\rho}_t (\theta , \omega )$.
Then, $Z_0(t, \omega ) = 1$ and $Z_j ,\, j = \pm 1, \pm 2, \cdots $ satisfy the 
differential equations
\begin{eqnarray*}
\frac{dZ_j}{dt} &=& \sqrt{-1}j\omega Z_j + \sqrt{-1} j K \sum^{\infty}_{-\infty}f_l Z^0_l(t) Z_{j-l} \\
&=& \sqrt{-1}j\omega Z_j +\sqrt{-1} j K f_j Z^0_j(t) + \sqrt{-1} j K \sum_{l\neq j} f_l Z^0_l(t) Z_{j-l}.
\end{eqnarray*}
The incoherent solution corresponds to the zero solution $Z_j \equiv 0$ for $j = \pm 1, \pm 2, \cdots $.
Since $|Z_j(t, \omega )| \leq 1$, $Z_j(t, \omega )$ is in the Hilbert space $L^2 (\mathbf{R}, g(\omega )d\omega )$
for every $t\,$ :
\begin{eqnarray*}
|| Z_j ||^2_{L^2 (\mathbf{R}, g(\omega )d\omega )} = \int_{\mathbf{R}} \! |Z_j(t, \omega )|^2 g(\omega )d\omega \leq 1.
\end{eqnarray*} 
Thus we linearize the above equation as an evolution equation on $L^2 (\mathbf{R}, g(\omega )d\omega )$
\begin{equation}
\frac{dZ_j}{dt} = \left( j \sqrt{-1} \mathcal{M} + j \sqrt{-1} K f_j \mathcal{P}\right) Z_j,\,\, 
j = \pm 1, \pm 2, \cdots ,
\label{4-1}
\end{equation}
where $\mathcal{M}: q(\omega ) \mapsto \omega q(\omega )$ is the multiplication operator 
on $L^2 (\mathbf{R}, g(\omega )d\omega )$ and $\mathcal{P}$ is the projection on $L^2 (\mathbf{R}, g(\omega )d\omega )$
defined to be
\begin{equation}
\mathcal{P}q(\omega ) = \int_{\mathbf{R}} \! q(\omega )g(\omega ) d\omega . 
\end{equation} 
If we put $P_0(\omega ) \equiv 1$, $\mathcal{P}$ is also written as
$\mathcal{P}q(\omega ) = (q, P_0)$, where $(\,\, , \,\,)$ 
is the inner product on $L^2 (\mathbf{R}, g(\omega )d\omega )$:
\begin{equation}
(q_1, q_2) := \int_{\mathbf{R}} \! q_1(\omega ) \overline{q_2(\omega )} g(\omega )d\omega .
\end{equation}
Note that the order parameter is given as $Z^0_1 = (Z_1, P_0)$.
To determine the stability of the incoherent solution and the order parameter, 
we have to investigate the spectrum and the semigroup of the operator
$T_j := j \sqrt{-1} \mathcal{M} + j \sqrt{-1} K f_j \mathcal{P}$.

\subsection{Analysis of the operator $\sqrt{-1}\mathcal{M}$}

If $f_j = 0$, $T_j = j\sqrt{-1}\mathcal{M}$.
It is known that the multiplication operator $\mathcal{M}$ on $L^2 (\mathbf{R}, g(\omega )d\omega )$
is self-adjoint and its spectrum is given by 
$\sigma (\mathcal{M}) = \mathrm{supp} (g) \subset \mathbf{R}$, where $\mathrm{supp} (g)$ is a support of the 
density function $g$. Thus the spectrum of $j\sqrt{-1} \mathcal{M}$ is 
\begin{equation}
\sigma (j\sqrt{-1} \mathcal{M})= j\sqrt{-1} \cdot \mathrm{supp} (g)
 = \{ j\sqrt{-1} \lambda \, | \, \lambda \in \mathrm{supp} (g) \} \subset \sqrt{-1}\mathbf{R}.
\label{4-3}
\end{equation}
The semi-group $e^{j\sqrt{-1}\mathcal{M}t}$ generated by $j\sqrt{-1}\mathcal{M}$\,
is given as $e^{j\sqrt{-1}\mathcal{M}t} q(\omega ) = e^{j\sqrt{-1}\omega t}q(\omega ) $.
In particular, we obtain
\begin{equation}
(e^{j\sqrt{-1}\mathcal{M}t} q_1, q_2)
 = \int_{\mathbf{R}} \! e^{j\sqrt{-1}\omega t}q_1(\omega )\overline{q_2(\omega )}g(\omega )d\omega  
\label{4-4}
\end{equation}
for any $q_1, q_2 \in L^2 (\mathbf{R}, g(\omega )d\omega)$.
This is the Fourier transform of the function $q_1(\omega ) \overline{q_2 (\omega )} g(\omega )$.
Thus if $q_1(\omega ) \overline{q_2 (\omega )} g(\omega )$ is real analytic on $\mathbf{R}$
and has an analytic continuation to a neighborhood of the real axis,
then $(e^{j\sqrt{-1}\mathcal{M}t} q_1, q_2) $ decays exponentially as $t\to \infty$,
while if $q_1(\omega ) q_2 (\omega ) g(\omega )$ is $C^r$, then it decays as $O(1/t^r)$
(see Vilenkin~\cite{Vil}).

These facts are summarized as follows:
\\[0.2cm]
\begin{prp}
 Suppose that $f_j=0$ and Eq.(\ref{4-1}) is reduced to
$dZ_j/dt = j \sqrt{-1} \mathcal{M} Z_j$.
A solution of this equation with an initial value $q(\omega ) \in L^2 (\mathbf{R}, g(\omega )d\omega)$
is given by $Z_j(t) = e^{j \sqrt{-1} \mathcal{M} t}q(\omega ) = e^{j \sqrt{-1} \omega  t}q(\omega )$.
In particular the linearized order parameter $Z^0_1 (t) = (e^{\sqrt{-1}\mathcal{M}t}q ,  P_0) $
decays exponentially as $t\to \infty$ if $g(\omega )$ and $q(\omega )$
have analytic continuations to a neighborhood of the real axis.
\end{prp}

The resolvent $(\lambda - j\sqrt{-1}\mathcal{M})^{-1}$ of the operator $j\sqrt{-1}\mathcal{M}$
is calculated as
\begin{equation}
(  (\lambda - j\sqrt{-1}\mathcal{M})^{-1}q_1, q_2 ) 
 = \int_{\mathbf{R}} \! \frac{1}{\lambda - j\sqrt{-1} \omega } q_1(\omega ) \overline{q_2 (\omega )}g(\omega )d\omega . 
\label{4-7}
\end{equation}
We define the function $D(\lambda )$ to be
\begin{equation}
D(\lambda ) = (  (\lambda - j\sqrt{-1}\mathcal{M})^{-1}P_0, P_0 ) 
 = \int_{\mathbf{R}} \! \frac{1}{\lambda - j\sqrt{-1} \omega }g(\omega )d\omega 
\label{4-8}
\end{equation}
(recall that $P_0(\omega ) \equiv 1$).
It is holomorphic in $\mathbf{C}\backslash \sigma (j\sqrt{-1}\mathcal{M})$
and will play an important role in the later calculation.

\subsection{Analysis of the operator $T_j = j\sqrt{-1}\mathcal{M} + j\sqrt{-1}Kf_j \mathcal{P}$}

In what follows, we suppose that $f_j \neq 0$.
The domain $\mathsf{D}(T_j)$ of $T_j$ is given by $\mathsf{D}(\mathcal{M}) \cap \mathsf{D}(\mathcal{P}) = \mathsf{D}(\mathcal{M})$.
Since $\mathcal{M}$ is self-adjoint and since $\mathcal{P}$ is bounded, $T_j$ is a closed operator \cite{Kato}.
Let $\mathfrak{\varrho} (T_j)$ be the resolvent set of $T_j$ and 
$\sigma (T_j) = \mathbf{C}\backslash \mathfrak{\varrho} (T_j)$ the spectrum.
Since $T_j$ is closed, there is no residual spectrum.
Let $\sigma _p(T_j)$ and $\sigma _c(T_j)$ be the point spectrum (the set of eigenvalues)
and the continuous spectrum of $T_j$, respectively.
\\[0.2cm]
\begin{prp} (i) Eigenvalues $\lambda $ of $T_j$ are given as roots of
\begin{equation}
D(\lambda ) = \frac{1}{j\sqrt{-1}Kf_j},\,\,\, \lambda \in \mathbf{C}\backslash \sigma (j\sqrt{-1}\mathcal{M}).
\label{4-10}
\end{equation} 
(ii) The continuous spectrum of $T_j$ is given by
\begin{equation}
\sigma _c(T_j) = \sigma (j\sqrt{-1} \mathcal{M}) = j\sqrt{-1}\cdot \mathrm{supp} (g).
\end{equation}
\end{prp}

\begin{proof} \, (i) Suppose that $\lambda \in \sigma _p(T_j) \backslash  \sigma (j\sqrt{-1} \mathcal{M})$.
Then, there exists $x\in L^2 (\mathbf{R}, g(\omega )d\omega) $ such that
\begin{eqnarray*}
\lambda x = (j\sqrt{-1}\mathcal{M} + j\sqrt{-1}Kf_j \mathcal{P}) x,\,\,\, x\neq 0.
\end{eqnarray*}
Since $\lambda \notin \sigma (j\sqrt{-1} \mathcal{M})$, $(\lambda - j\sqrt{-1} \mathcal{M})^{-1}$
is defined and the above is rewritten as
\begin{eqnarray*}
x &=& (\lambda - j\sqrt{-1} \mathcal{M})^{-1} j\sqrt{-1}Kf_j \mathcal{P}x  \\
&=& j\sqrt{-1}Kf_j ( x, P_0 )  (\lambda - j\sqrt{-1} \mathcal{M})^{-1} P_0(\omega ).
\end{eqnarray*}
By taking the inner product with $P_0(\omega )$, we obtain
\begin{equation}
1 = j\sqrt{-1}Kf_j (  (\lambda - j\sqrt{-1} \mathcal{M})^{-1}P_0,   P_0 )  = j\sqrt{-1}Kf_j  D(\lambda ).
\end{equation}
This proves that roots of Eq.(\ref{4-10}) is in $\sigma _p(T_j)$.
The corresponding eigenvector is given by $x = (\lambda - j\sqrt{-1} \mathcal{M})^{-1} P_0(\omega )
 = 1/(\lambda - j \sqrt{-1}\omega )$.
If $\lambda \in \sqrt{-1}\mathbf{R}$, $x\notin L^2 (\mathbf{R}, g(\omega )d\omega)$.
Thus there are no eigenvalues on the imaginary axis.
\\
(ii) This follows from the fact that the essential spectrum is stable under the 
bounded perturbation and that there are no eigenvalues on $\sigma (j\sqrt{-1} \mathcal{M})$, see \cite{Kato}.
\end{proof}

\subsection{Eigenvalues of the operator $T_j$ and the transition point formula}

Our next task is to calculate roots of Eq.(\ref{4-10}) to obtain eigenvalues of 
$T_j = j\sqrt{-1} \mathcal{M} + j\sqrt{-1}Kf_j \mathcal{P}$.
By putting $\lambda = x + \sqrt{-1}y,\,\, x,y\in \mathbf{R}$,
Eq.(\ref{4-10}) is rewritten as
\begin{equation}
\left\{ \begin{array}{l}
\displaystyle 
\int_{\mathbf{R}} \! \frac{x}{x^2 + (j\omega -y)^2}g(\omega )d\omega 
    = - \frac{\mathrm{Im} (f_j)}{jK|f_j|^2},\\[0.4cm]
\displaystyle 
\int_{\mathbf{R}} \! \frac{j\omega -y}{x^2 + (j\omega -y)^2}g(\omega )d\omega 
    = - \frac{\mathrm{Re} (f_j)}{jK|f_j|^2}. \\
\end{array} \right.
\label{4-14}
\end{equation}
In what follows, we suppose that $\mathrm{Im} (f_j) < 0$.
The case $\mathrm{Im} (f_j) \geq 0$ will be treated in Sec.3.5.
The next lemma is easily obtained.
\\[0.2cm]
\begin{lmm}
\\
(i) When $\mathrm{Im} (f_j) <0$, $\lambda $ satisfies $\mathrm{Re} (\lambda ) > 0$ for any $K > 0$.
\\
(ii) If $K>0$ is sufficiently large, there exists at least one eigenvalue $\lambda $ near infinity.
\\
(iii) If $K>0$ is sufficiently small, there are no eigenvalues.
\end{lmm}

\begin{proof}  Part (i) of the lemma immediately follows from the first equation of Eq.(\ref{4-14}).
To prove part (ii) of the lemma, note that if $|\lambda |$ is large, Eq.(\ref{4-10}) is rewritten as
\begin{eqnarray*}
\frac{1}{\lambda } + O(\frac{1}{\lambda ^2}) = \frac{1}{j\sqrt{-1} K f_j}.
\end{eqnarray*}
Thus the Rouch\'{e}'s theorem proves that Eq.(\ref{4-10}) has a root $\lambda \sim j\sqrt{-1} K f_j$ if $K>0$
is sufficiently large.
To prove part (iii) of the lemma, we see that the left hand side of the first equation of 
Eq.(\ref{4-14}) is bounded for any $x, y \in \mathbf{R}$.
To do so, let $G(\omega )$ be the primitive function of $g(\omega )$
and fix $\delta >0$ small. The left hand side of the first equation of Eq.(\ref{4-14}) is calculated as
\begin{eqnarray*}
& & \int_{\mathbf{R}} \! \frac{xg(\omega )d\omega}{x^2 + (j\omega -y)^2} \\
 &=& \int^\infty_{y/j + \delta } \! \frac{xg(\omega )d\omega}{x^2 + (j\omega -y)^2}
 + \int^{y/j-\delta }_{-\infty} \! \frac{xg(\omega )d\omega}{x^2 + (j\omega -y)^2}
 + \int^{y/j + \delta }_{y/j - \delta } \! \frac{xg(\omega )d\omega}{x^2 + (j\omega -y)^2} \\
&=& \int^\infty_{y/j + \delta } \! \frac{xg(\omega )d\omega}{x^2 + (j\omega -y)^2}
 + \int^{y/j-\delta }_{-\infty} \! \frac{xg(\omega )d\omega}{x^2 + (j\omega -y)^2} \\
& & \quad + \frac{x}{x^2 + j^2\delta ^2} \left( G(y/j + \delta ) - G(y/j - \delta )\right)
 + \int^{y/j + \delta }_{y/j - \delta } \! \frac{2jx (j\omega -y)}{(x^2 + (j\omega -y)^2)^2} G(\omega )d\omega .
\end{eqnarray*}
The first three terms in the right hand side above are bounded for any $x, y\in \mathbf{R}$.
Since $G$ is continuous, there exists a number $\xi$ such that
the last term is rewritten as
\begin{eqnarray*}
\int^{y/j + \delta }_{y/j - \delta } \! \frac{2jx (j\omega -y)}{(x^2 + (j\omega -y)^2)^2} G(\omega )d\omega
 = 2j\delta \cdot \frac{2x \xi}{(x^2 + \xi^2)^2} G(y/j + \xi /j).
\end{eqnarray*}
This is bounded for any $x, y\in \mathbf{R}$.
Now we have proved that the left hand side of the first equation of Eq.(\ref{4-14})
 is bounded for any $x > 0$, although the right hand side tends to infinity as $K\to +0$.
Thus Eq.(\ref{4-10}) has no roots if $K$ is small. 
\end{proof}

Lemma 3.3 shows that if $K>0$ is sufficiently large, the trivial solution $Z_j = 0$ of the system $dZ_j/dt = T_jZ_j$
is unstable because of the eigenvalues with positive real parts.
Our purpose in this subsection is to determine the bifurcation point $K_c^{(j)}$, which is the minimum value of $K$
such that if $K < K_c^{(j)}$, the operator $T_j$ has no eigenvalues on the right half plane.
To calculate eigenvalues $\lambda = \lambda (K)$ explicitly is difficult in general.
However, note that
since zeros of a holomorphic function do not vanish because of the argument principle,
$\lambda (K)$ disappears if and only if it is absorbed into the continuous spectrum $\sigma (j\sqrt{-1}\mathcal{M})$,
on which $D(\lambda )$ is not holomorphic.
This fact suggests that to determine $K_c^{(j)}$, it is sufficient to investigate Eq.(\ref{4-10}) or Eq.(\ref{4-14})
near the imaginary axis.
Since we are interested in $\lambda (K)$ absorbed into $\sigma (j\sqrt{-1}\mathcal{M}) \subset \sqrt{-1} \mathbf{R}$,
take the limit $x \to +0$ in Eq.(\ref{4-14}):
\begin{equation}
\left\{ \begin{array}{l}
\displaystyle 
\lim_{x\to +0}\int_{\mathbf{R}} \! 
\frac{x}{x^2 + (j\omega -y)^2}g(\omega )d\omega  = - \frac{\mathrm{Im} (f_j)}{jK|f_j|^2},\\[0.4cm]
\displaystyle 
\lim_{x\to +0}\int_{\mathbf{R}} \! 
\frac{j\omega -y}{x^2 + (j\omega -y)^2}g(\omega )d\omega = - \frac{\mathrm{Re} (f_j)}{jK|f_j|^2}. \\
\end{array} \right.
\label{4-15}
\end{equation}
These equations determine $K_n$ and $y_n$ such that one of the eigenvalues $\lambda _n(K)$
converges to $\sqrt{-1} y_n$
as $K\to K _n + 0$ (see Fig.\ref{fig3}). To calculate them, we need the next lemma.

\begin{figure}
\begin{center}
\includegraphics[]{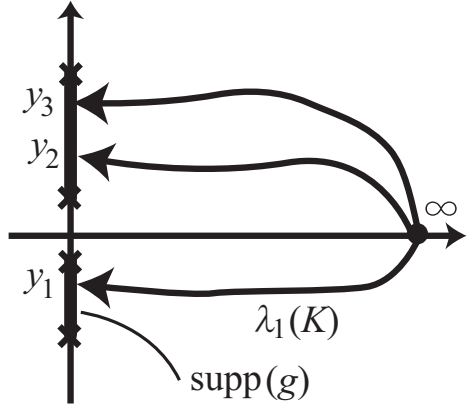}
\caption[]{A schematic view of behavior of roots $\lambda $ of Eq.(\ref{4-10}) when $K$ decreases.}
\label{fig3}
\end{center}
\end{figure}

\begin{lmm} (i) Suppose that $\lambda _n(K) \to \sqrt{-1}y_n$ as $K\to K_n$.
Then, $g(\omega )$ is continuous at $\omega = y_n$.
\\
(ii) If $g(\omega )$ is continuous at $\omega = y$, then
\begin{equation}
\lim_{x\to +0}\int_{\mathbf{R}} \! 
\frac{x}{x^2 + (j\omega -y)^2}g(\omega )d\omega = \pi g(y/j )/j.
\label{4-16}
\end{equation}
\end{lmm}

\begin{proof} 
To prove (i), suppose that $g(\omega )$ is discontinuous at $\omega =0$
without loss of generality.

STEP 1:  At first, we suppose that $g(\omega )$ is piecewise continuous.
Put $g(+0) = h_+,\, g(-0) = h_-$ and $h_+ \neq h_-$.
In this case, for any $\varepsilon >0$, there exists $\delta >0$ such that if $-\delta < \omega <0$,
then $|g(\omega ) - h_-| < \varepsilon $ and if $0 < \omega < \delta$,
then $|g(\omega ) - h_+| < \varepsilon $. 
For Eq.(\ref{4-10}), we suppose 
$|\lambda | = |x + \sqrt{-1}y| < \delta $ and $y >0$. The case $y < 0$ is treated in a similar manner.
We calculate $D(\lambda )$ as
\begin{eqnarray}
D(\lambda ) &=& \int^\infty_{\delta} \! \frac{g(\omega )}{\lambda - j\sqrt{-1}\omega } d\omega 
+ \int^{-\delta}_{-\infty} \! \frac{g(\omega )}{\lambda - j\sqrt{-1}\omega } d\omega \nonumber \\
& & + \int^\delta_{0} \! \frac{1}{\lambda -j\sqrt{-1}\omega }\left( g(\omega ) - h_+ + h_+ \right) d\omega 
 + \int^0_{-\delta} \! \frac{1}{\lambda -j\sqrt{-1}\omega }\left( g(\omega ) - h_- + h_- \right) d\omega \nonumber \\
&=& \int^\infty_{\delta} \! \frac{g(\omega )}{\lambda - j\sqrt{-1}\omega } d\omega 
+ \int^{\infty}_{\delta} \! \frac{g(-\omega )}{\lambda + j\sqrt{-1}\omega } d\omega \nonumber \\
& & + \int^\delta_{0} \! \frac{1}{\lambda -j\sqrt{-1}\omega }\left( g(\omega ) - h_+ \right) d\omega
 + \int^\delta_{0} \! \frac{1}{\lambda + j\sqrt{-1}\omega }\left( g(-\omega ) - h_- \right) d\omega \nonumber \\
& & + \int^\delta_{0} \! \frac{1}{\lambda -j\sqrt{-1}\omega }\left( h_+ - h_- \right) d\omega
 + h_- \int^\delta_{0} \!  \frac{d\omega}{\lambda -j\sqrt{-1}\omega }  
  + h_- \int^\delta_{0} \!  \frac{d\omega}{\lambda + j\sqrt{-1}\omega }.
\end{eqnarray}
Since $|\lambda | < \delta $, there exists a positive number $M$, which is independent of $\lambda $,
such that 
\begin{eqnarray*}
\int^\infty_{\delta} \! \frac{g(\pm \omega )}{\lambda \mp j\sqrt{-1}\omega } d\omega < M.
\end{eqnarray*}
Thus $|D(\lambda )|$ is estimated as
\begin{eqnarray*}
|D(\lambda )| & \geq & |h_+ - h_-| \int^\delta_{0} \! \frac{d\omega }{|\lambda -j\sqrt{-1}\omega |}
- \varepsilon  \int^\delta_{0} \! \frac{d\omega }{|\lambda -j\sqrt{-1}\omega |}
  - \varepsilon \int^\delta_{0} \! \frac{d\omega }{|\lambda + j\sqrt{-1}\omega |} \\ 
& & \quad - \left| \frac{h_-}{j} \log \left( \frac{\sqrt{-1} \lambda  + j\delta}{ \sqrt{-1}\lambda - j\delta }\right) \right| -2M.
\end{eqnarray*}
Since $y >0$, $|\lambda -j\sqrt{-1}\omega | < |\lambda + j\sqrt{-1}\omega |$.
This shows that 
\begin{equation}
|D(\lambda )| \geq \left( |h_+ - h_-| - 2\varepsilon \right)\int^\delta_{0} \! \frac{d\omega }{|\lambda -j\sqrt{-1}\omega |}
 - \left| \frac{h_-}{j} \log \left( \frac{\sqrt{-1} \lambda  + j\delta}{ \sqrt{-1}\lambda - j\delta }\right) \right| -2M.
\end{equation}
The right hand side tends to infinity as $\lambda \to 0$ if $2\varepsilon < |h_+ - h_-|$.
This proves that Eq.(\ref{4-10}) has no roots at $\lambda =0$ for positive $K$.

STEP 2: In general, since $g(\omega )$ is a non-negative measurable function, 
there exists a monotonic increasing sequence $\{ g_n (\omega )\}_{n=1}^\infty$ of non-negative simple functions
such that $g_n (\omega ) \to g (\omega )$ for each $\omega $.
In particular if $g(\omega )$ is discontinuous at $\omega  = 0$,
we can choose $\{ g_n (\omega )\}_{n=1}^\infty$ so that $g_n (\omega )$ is discontinuous at $\omega  = 0$
for any $n\in \mathbf{N}$.
Then, the proof is done in the same way as STEP 1 by approximating $g(\omega )$ by $g_n (\omega )$.

(ii) The formula Eq.(\ref{4-16}) is proved in Ahlfors~\cite{Ahl}. 
\end{proof} 

Let $(y, K)$ be one of the solutions of Eq.(\ref{4-15}).
Since $g(\omega )$ is continuous at $\omega  = y$, substituting it into the first equation
of Eq.(\ref{4-15}) yields
\begin{equation}
\pi g(y/j) = - \frac{\mathrm{Im}(f_j)}{K |f_j|^2}.
\end{equation}
Substituting $K = -\mathrm{Im }(f_j)/ (\pi |f_j|^2 g(y/j))$ obtained from the above
into the second equation of Eq.(\ref{4-15}) results in
\begin{equation}
\lim_{x\to 0}\int_{\mathbf{R}} \! 
\frac{j\omega -y}{x^2 + (j\omega -y)^2}g(\omega )d\omega 
  = \frac{\pi \mathrm{Re} (f_j)}{j \,\mathrm{Im}(f_j)} g(y/j).
\label{4-17}
\end{equation}
This equation for $y$ determines imaginary parts to which $\lambda (K)$ converges as
$\mathrm{Re} (\lambda (K)) \to +0$.
Let $y_1 , y_2, \cdots$ be roots of Eq.(\ref{4-17}).
Then, 
\begin{equation}
K_n = \frac{-\mathrm{Im}(f_j)}{\pi |f_j|^2 g(y_n/j)},\,\, n = 1,2,  \cdots 
\label{4-18}
\end{equation}
give the values such that $\mathrm{Re} (\lambda (K)) \to 0$ as $K \to K_n + 0$.
Now we obtain the next theorem.
\\[0.2cm]
\begin{thm}
 Suppose that $\mathrm{Im}(f_j) < 0$.
Let $y_1, y_2, \cdots$ be roots of Eq.(\ref{4-17}).
Put
\begin{equation}
K^{(j)}_c:= \inf_{n} K_n 
= \frac{-\mathrm{Im}(f_j)}{\pi |f_j|^2 \sup_{n} g(y_n/j)}.
\label{4-19}
\end{equation}
If $0 < K < K^{(j)}_c$, the operator $T_j$ has no eigenvalues, while if $K$ is slightly larger than
$K^{(j)}_c$, $T_j$ has eigenvalues on the right half plane.
\end{thm}

Note that $\inf_n K_n$ is positive because of Lemma.3.3 (iii).
As a corollary, we obtain the transition point (bifurcation point to the partially locked state)
conjectured by Kuramoto~\cite{Kura2}:
\\[0.2cm]
\begin{crl}[Kuramoto's transition point]
Suppose that the probability density function $g(\omega )$ is even and $\max g(\omega ) = g(0)$.
If $\mathrm{Re} (f_1) = 0$ and $\mathrm{Im} (f_1) = -1/2$ (it corresponds to $f(\theta ) = \sin \theta $
in Eq.(\ref{KDMN})), then $K^{(1)}_c$ defined as above is given by
\begin{equation}
K^{(1)}_c = \frac{2}{\pi g(0)}.
\label{4-20}
\end{equation}
When $K > K^{(1)}_c$, the solution $Z_1 = 0$ of Eq.(\ref{4-1}) is unstable.
\end{crl}


\subsection{Semi-group generated by the operator $T_1$ ($\mathrm{Im} (f_1) <0$)}

Since we are interested in the dynamics of the order parameter $Z^0_1(t) = (Z_1, P_0)$,
in what follows, we consider only $j=1$ while cases $j=2,3,\cdots $ are investigated
in the same way.
Theorem 3.5 shows that $K_c^{(1)}$ is the least bifurcation point and the trivial solution 
$Z_1(t, \omega ) = 0$ of Eq.(\ref{4-1}) is unstable if $K$ is slightly larger than $K_c^{(1)}$.
If $0 < K <  K^{(1)}_c$, the spectrum of $T_1$ is on the imaginary axis: 
$\sigma (T_1) = \sigma (\sqrt{-1}\mathcal{M})$, and thus the dynamics of 
$Z_1$ is nontrivial.
In this subsection, we investigate the dynamics of $Z_1$ and 
the order parameter for $0 < K <  K^{(1)}_c$.
We will see that the order parameter 
may decay exponentially even if the spectrum lies on the imaginary axis
because of existence of resonance poles.

Since $\sqrt{-1}\mathcal{M}$ has the semi-group $e^{\sqrt{-1}\mathcal{M}t}$ and since $\mathcal{P}$
is bounded, the operator $T_1 = \sqrt{-1} \mathcal{M} + \sqrt{-1} K f_1 \mathcal{P}$ also generates
the semi-group (Kato~\cite{Kato}), say $e^{T_1t}$.
A solution of Eq.(\ref{4-1}) with an initial value $q(\omega ) \in L^2 (\mathbf{R}, g(\omega )d\omega)$
is given by $e^{T_1t}q(\omega )$.
The $e^{T_1t}$ is calculated by using the Laplace inversion formula
\begin{equation}
e^{T_1t} = \lim_{y\to \infty} \frac{1}{2\pi \sqrt{-1}} \int^{x + \sqrt{-1}y}_{x-\sqrt{-1}y} \!
e^{\lambda t} (\lambda -T_1)^{-1} d\lambda , 
\label{semi1}
\end{equation}
where $x > 0$ is chosen so that the contour is to the right of the spectrum of $T_1$ (Yosida~\cite{Yos}).
At first, let us calculate the resolvent $(\lambda -T_1)^{-1}$.
\\[0.2cm]
\begin{lmm}
For any $q(\omega ) \in L^2 (\mathbf{R}, g(\omega )d\omega)$, the equality
\begin{equation}
F_0(\lambda ) := (  (\lambda -T_1)^{-1} q, P_0) 
 = \frac{( (\lambda - \sqrt{-1} \mathcal{M})^{-1} q, P_0 ) }
{1 - \sqrt{-1}Kf_1 D(\lambda )}
\label{semi2}
\end{equation}
holds.
\end{lmm}

\begin{proof} 
Put
\begin{eqnarray*}
R(\lambda )q := (\lambda -T_1)^{-1} q = (\lambda - \sqrt{-1}\mathcal{M}
 - \sqrt{-1}Kf_1 \mathcal{P})^{-1}q,
\end{eqnarray*}
which yields
\begin{eqnarray*}
(\lambda - \sqrt{-1}\mathcal{M})R(\lambda )q
 &=& q + \sqrt{-1}Kf_1\mathcal{P}R(\lambda )q \\
&=&  q + \sqrt{-1}Kf_1 (R(\lambda )q, P_0) P_0.
\end{eqnarray*}
This is rearranged as
\begin{eqnarray*}
R(\lambda )q =  (\lambda - \sqrt{-1}\mathcal{M})^{-1}q 
    + \sqrt{-1}Kf_1 (R(\lambda )q, P_0) (\lambda - \sqrt{-1}\mathcal{M})^{-1}P_0.
\end{eqnarray*}
By taking the inner product with $P_0$, we obtain
\begin{eqnarray*}
(R(\lambda )q, P_0) =  ((\lambda - \sqrt{-1}\mathcal{M})^{-1}q, P_0)
               + \sqrt{-1}Kf_1 (R(\lambda )q, P_0) D(\lambda ).
\end{eqnarray*}
This proves Eq.(\ref{semi2}).
\end{proof} 

Let $Z^0_1(t) = ( Z_1, P_0) $ be the order parameter with the initial condition $Z_1(0, \omega ) = q(\omega )$.
Eqs.(\ref{semi1}) and (\ref{semi2}) show that $Z^0_1(t)$ is given by
\begin{equation}
Z^0_1(t) =
(  e^{T_1t}q ,P_0) 
 = \lim_{y\to \infty} \frac{1}{2\pi \sqrt{-1}} \int^{x + \sqrt{-1}y}_{x-\sqrt{-1}y} \!
e^{\lambda t} \frac{( (\lambda - \sqrt{-1} \mathcal{M})^{-1} q , P_0) }
{1 - \sqrt{-1}Kf_1 D(\lambda )} d\lambda.
\label{semi3}
\end{equation}
One of the effective way to calculate the integral above is to use the residue theorem.
Recall that the resolvent $(\lambda - T_1)^{-1}$ is holomorphic on $\mathbf{C} \backslash \sigma (T_1)$.
Since we assume that $0 < K < K^{(1)}_c$, $T_1$ has no eigenvalues and the continuous 
spectrum lies on the imaginary axis : 
$\sigma (T_1) = \sigma (\sqrt{-1} \mathcal{M}) = \sqrt{-1} \cdot \mathrm{supp} (g)$.
Thus the integrand $e^{\lambda t} F_0(\lambda )$ in Eq.(\ref{semi3}) is holomorphic
on the right half plane and may not be holomorphic on $\sigma (T_1)$.
However, under assumptions below, we can show that $F_0(\lambda )$ has an analytic continuation $F_1(\lambda )$
through the line $\sigma (T_1)$ from right to left.
Then, $F_1(\lambda )$ may have poles on the left half plane (the second Riemann sheet
of the resolvent), which are called \textit{resonance poles}~\cite{Reed}.
The resonance pole $\mu$ affects the integral in Eq.(\ref{semi3}) through the residue theorem
(see Fig.\ref{fig6}). In this manner, the order parameter $Z^0_1(t)$ can decay with the exponential rate $\mathrm{Re} (\mu)$.
Such an exponential decay caused by resonance poles is well known in the theory of Schr\"{o}dinger operators \cite{Reed},
and for the Kuramoto model, it is investigated numerically by Strogatz \textit{et al.}~\cite{Str2}
and Balmforth \textit{et al.}~\cite{Bal}.

At first, we construct an analytic continuation of the function $F_0(\lambda )$.
\\[0.2cm]
\begin{lmm}
Suppose that the probability density function $g(\omega )$
and an initial condition $q(\omega )$ are real analytic on $\mathbf{R}$.
If $g(\omega )$ and $q(\omega)$ have meromorphic continuations $g^*(\lambda )$ and $q^*(\lambda )$
to the upper half plane, respectively,
then the function $F_0(\lambda )$ defined on the right half plane has the meromorphic continuation
$F_1(\lambda )$ to the left half plane, which is given by
\begin{equation}
F_1(\lambda ) = \frac{( (\lambda - \sqrt{-1} \mathcal{M})^{-1} q,  P_0) 
 + 2\pi q^* (-\sqrt{-1}\lambda ) g^*(-\sqrt{-1}\lambda )}
{1 - \sqrt{-1}Kf_1 D(\lambda ) - 2\pi \sqrt{-1}Kf_1 g^*(-\sqrt{-1}\lambda )}.
\label{semi4}
\end{equation}
\end{lmm}

\begin{proof} 
By the formula (\ref{4-16}), we obtain
\begin{equation}
\lim_{\mathrm{Re}(\lambda ) \to +0} ((\lambda - \sqrt{-1} \mathcal{M})^{-1} q ,  P_0) 
  - \lim_{\mathrm{Re}(\lambda ) \to -0} ((\lambda - \sqrt{-1} \mathcal{M})^{-1} q ,  P_0) 
 = 2\pi q (\mathrm{Im} (\lambda )) \cdot g(\mathrm{Im} (\lambda )).
\end{equation}
Thus the meromorphic continuation of $( (\lambda - \sqrt{-1} \mathcal{M})^{-1} q , P_0)$
from right to left is given by
\begin{equation}
\left\{ \begin{array}{ll}
((\lambda - \sqrt{-1} \mathcal{M})^{-1} q ,  P_0)  & (\mathrm{Re} (\lambda ) > 0),  \\
( (\lambda - \sqrt{-1} \mathcal{M})^{-1} q , P_0)  + 2\pi q^* (-\sqrt{-1}\lambda ) g^* (-\sqrt{-1}\lambda )
   & (\mathrm{Re} (\lambda ) < 0). \\
\end{array} \right.
\end{equation}
This proves Eq.(\ref{semi4}). 
\end{proof} 

Poles of $F_1(\lambda )$ (resonance poles) on the left half plane are given as roots of the equation
\begin{equation}
D(\lambda ) + 2\pi g^* (-\sqrt{-1} \lambda ) = \frac{1}{\sqrt{-1}K f_1}, \quad \mathrm{Re}(\lambda ) < 0
\label{semi7}
\end{equation}
and poles of the function $q^* (-\sqrt{-1}\lambda )$.
In the next theorem, we suppose for simplicity that $q^* (-\sqrt{-1}\lambda )$ has no poles.
Now we calculate the order parameter $Z^0_1(t)$.
\\[0.2cm]
\begin{thm}
For Eq.(\ref{4-1}) with $j=1$, suppose that
\\
(i)\, $\mathrm{Im} (f_1) < 0$ and $0 < K < K^{(1)}_c$.
\\
(ii) \, the probability density function $g(\omega )$ is real analytic
on $\mathbf{R}$ and has a meromorphic continuation $g^* (\lambda )$ to the upper half plane.
\\
(iii) \, an initial condition $q(\omega )$ is real analytic
on $\mathbf{R}$ and has an analytic continuation $q^*(\lambda )$
to the upper half plane.
\\
(iv)\, there exists a positive number $\delta $ such that $|F_1(\lambda )| \to 0$ as $|\lambda | \to \infty$
in the angular domains
\begin{equation}
|\mathrm{arg} (\lambda ) | \leq \delta ,\,\,\,
| \mathrm{arg} (\lambda ) -\pi | \leq \delta.
\label{angular}
\end{equation}
(v)\, there exist positive constants $D$ and $\beta$ such that
\begin{equation}
|F_1(\lambda )| \leq D e^{\beta |\lambda |}
\end{equation}
in the angular domain $\pi/2 + \delta \leq \mathrm{arg} (\lambda ) \leq 3\pi/2 - \delta $.
\\[0.2cm]
Then, there exist resonance poles of $T_1$ on the left half plane.
Let $\alpha _1, \alpha _2, \cdots $ be resonance poles 
such that $|\alpha _1| \leq |\alpha _2| \leq \cdots $.
Then, there exists a positive constant $t_0$ such that the order parameter is given by
\begin{equation}
Z^0_1(t) = (e^{T_1t}q,  P_0) = \sum^\infty_{n=1} p_n(t)e^{\alpha _n t}, \quad t>t_0
\label{semi8}
\end{equation}
where $p_n(t)$ is a polynomial in $t$.
In particular, $Z^0_1(t)$ decays exponentially as $t\to \infty$.
\end{thm}

\begin{proof}
At first, we prove the existence of resonance poles.
Resonance poles are roots of Eq.(\ref{semi7}), which is the analytic continuation 
of the equation (\ref{4-10}) for $j=1$. Thus one of the resonance poles is obtained 
as a continuation of an eigenvalue $\lambda (K)$.
Recall that $\lambda (K)$ converges into the imaginary axis as
$K\to K^{(1)}_c + 0$. To prove that there exists a resonance pole on the left half plane
when $K < K_c^{(1)}$, we have to show that $\lambda (K)$ does not stay on the imaginary axis for $K < K^{(1)}_c$.
Differentiating Eq.(\ref{4-10}) with respect to $K$, we obtain
\begin{equation}
\lambda '(K) \int_{\mathbf{R}} \! \frac{1}{(\lambda - \sqrt{-1}\omega )^2}g(\omega )d\omega  = \frac{1}{\sqrt{-1}K^2f_1}, 
\end{equation}
which proves that $\lambda '(K) \neq 0$.
Further, roots $y$ of Eq.(\ref{4-17}), which determines eigenvalues on the imaginary axis,
are isolated because both side of Eq.(\ref{4-17}) are analytic
with respect to $y$.
This means that $\lambda (K)$ can not move along the imaginary axis.
This proves that an eigenvalue $\lambda (K)$ gets across the imaginary axis 
from right to left as $K$ decreases from $ K^{(1)}_c$, which gives a root of Eq.(\ref{semi7}).
Note that there may exist resonance poles which are not continuations of eigenvalues (see Example 3.11).

Next, let us prove Eq.(\ref{semi8}).
Let $d >0$ be a small number and $r$ sufficiently large number.
Take paths $C_1$ to $C_6$ as are shown in Fig.\ref{fig6}:
\begin{eqnarray*}
& & C_1 = \{ d + \sqrt{-1} y \, | \, -r \leq y \leq r\}, \\
& & C_2 = \{ x + \sqrt{-1}r \, | \, 0 \leq x \leq d\}, \\
& & C_3 = \{ r e^{\sqrt{-1}\theta } \, | \, \pi /2 \leq \theta \leq \pi / 2 + \delta  \}, \\
& & C_4 = \{ r e^{\sqrt{-1}\theta } \, | \, \pi / 2 + \delta \leq \theta \leq 3\pi / 2 - \delta  \},
\end{eqnarray*}
and $C_5$ and $C_6$ are defined in a similar way to $C_3$ and $C_2$, respectively.
We put $C = \sum^6_{j=1}C_j$.

\begin{figure}
\begin{center}
\includegraphics[]{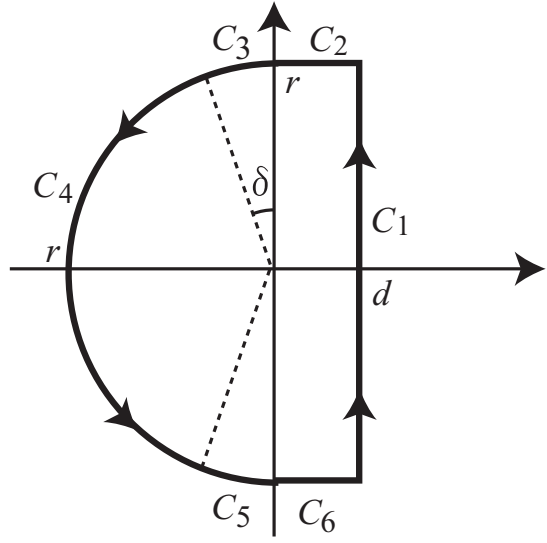}
\caption[]{The contour for the Laplace inversion formula.}
\label{fig6}
\end{center}
\end{figure}

Let $\alpha _1 , \alpha _2 ,\cdots ,\alpha _{n(C)}$ be resonance poles 
inside the closed curve $C$, where we assume that there are no resonance poles
on the curve $C$ by deforming it slightly if necessary.
Let $R_1(t) , R_2(t)$, $\cdots  ,R_{n(c)}(t)$ be corresponding residues of $e^{\lambda t} F_1(\lambda )$,
respectively. Note that if $\alpha _j$ is a pole of $F_1(\lambda )$ of order $m_j$,
$R_j(t)$ is of the form $R_j(t) = p_j(t) e^{\alpha _j t}$ with a polynomial $p_j(t)$ of degree $m_j-1$.
By the residue theorem, we have
\begin{eqnarray*}
2\pi \sqrt{-1} \sum^{n(C)}_{j=1} R_j(t)
 = \int_{C_6 + C_1 + C_2} \! e^{\lambda t} F_0(\lambda ) d\lambda 
  + \int_{C_3 + C_4 + C_5} \! e^{\lambda t} F_1(\lambda ) d\lambda.
\end{eqnarray*}
The integral $\int_{C_1} \! e^{\lambda t} F_0(\lambda ) d\lambda / (2\pi \sqrt{-1})$ converges to $Z^0_1(t)$
as $r \to \infty$.
It is easy to show that the integrals along $C_2, C_3, C_5, C_6$ tend to zero as $r\to \infty$
because of the assumption (iv).
We have to estimate the integral along $C_4$ as
\begin{eqnarray}
\left| \int_{C_4} \! e^{\lambda t} F_1(\lambda ) d\lambda \right| 
&\leq & \int^{3\pi /2 - \delta }_{\pi /2 + \delta }
  \! r e^{r t \cos \theta } \, |F_1 (r e^{\sqrt{-1} \theta })| d\theta  \nonumber \\
&\leq & \max_{\pi / 2 + \delta \leq \theta  \leq 3\pi /2 - \delta } |F_1 (r e^{\sqrt{-1}\theta })|
\int^{3\pi /2 - \delta }_{\pi /2 + \delta }
  \! r e^{r t \cos \theta } d\theta  \nonumber \\
&\leq & D e^{\beta r}
\int^{\pi /2 }_{ \delta } \! 2r e^{-r t \sin \phi} d\phi  \nonumber \\
&\leq & D e^{\beta r}
\int^{\pi /2 }_{ \delta } \! 2r e^{-2r t \phi /\pi} d\phi  \nonumber \\
& \leq &  D e^{\beta r}
   \cdot \frac{\pi }{t} \left( e^{-2r t \delta / \pi} - e^{-r t}\right) .
\label{semi9}
\end{eqnarray}
Thus if $t > t_0 := \max \,\{ \beta, \pi \beta/(2\delta )\}$, 
this integral tends to zero as $r\to \infty$.  
\end{proof}

\begin{exa}
If $g(\omega )$ is a rational function,
the assumptions are satisfied when $q^*(\lambda )$ is bounded on the upper half plane.
In this case, the number of resonance poles is finite and thus
Eq.(\ref{semi8}) becomes finite sum.
For example if $g(\omega ) = 1/(\pi (1 + \omega ^2))$ is the Lorentzian distribution,
a resonance pole is given by $\lambda  = \sqrt{-1}Kf_1 -1$ (a root of Eq.(\ref{semi7})).
Therefore $Z^0_1(t)$ decays with the exponential rates $\mathrm{Re} (\sqrt{-1}Kf_1 -1)$.
\end{exa}

\begin{exa}
If $g(\omega )$ is the Gaussian distribution,
the assumptions are satisfied when $q^*(\lambda )$ is of exponential type; that is, 
there exist positive constants $C$ and $\beta$ such that $|q^*(\lambda )| \leq Ce^{\beta |\lambda |}$.
Since the analytic continuation $g^*(\lambda )$ has an essential singularity at infinity, 
there exist infinitely many resonance poles and they accumulate at infinity.
\end{exa}


\subsection{Semi-group generated by the operator $T_1$ ($\mathrm{Im} (f_1) \geq 0$)}

In Sec.3.1 and Sec.3.4, we investigate the semi-group generated by the operator 
$T_1 = \sqrt{-1} \mathcal{M} + \sqrt{-1}Kf_1 \mathcal{P}$ for the cases $f_1 = 0$ and $\mathrm{Im} (f_1) < 0$,
respectively.
In this subsection, we consider the case $\mathrm{Im} (f_1) \geq 0$.
\\[0.2cm]
\textbf{Theorem 3.12.}\, Suppose that the assumptions (ii) to (v) of Thm.3.9 hold.
If $\mathrm{Im} (f_1) \geq 0$, for an arbitrarily fixed $K>0$,
the order parameter $Z^0_1(t) = (e^{T_1t}q, P_0)$ decays exponentially as $t\to \infty$.
\\[0.2cm]
We show an idea of the proof.
If $\mathrm{Im} (f_1) = 0$, $T_1/\sqrt{-1} = \mathcal{M} + K \mathrm{Re}(f_1) \mathcal{P}$
is self-adjoint and a rank one perturbation of the multiplication $\mathcal{M}$.
By Theorem X-4.3 in \cite{Kato}, $T_1/\sqrt{-1}$ and $\mathcal{M}$ are unitarily equivalent.
Since $(e^{\sqrt{-1}\mathcal{M}t}q, P_0)$ decays exponentially (see Sec.3.1),
we can prove that so is $(e^{T_1t}q, P_0)$.

If $\mathrm{Im} (f_1) > 0$, change the parameter as $K \mapsto -K$.
Then, the problem is reduced to the case $K<0$ and $\mathrm{Im} (f_1) < 0$, and Thm.3.12
is proved in a similar manner to the proof of Thm.3.9.


\end{document}